\documentclass{amsart}
\usepackage{amssymb,amsthm,amscd}

\newcommand{\rsp}{\raisebox{0em}[2ex][1.3ex]{\rule{0em}{2ex} }}

\newcommand{\Z}{{\mathbb Z}}

\newcommand{\Q}{{\mathbb Q}}

\newcommand{\fa}{\mathfrak a}
\newcommand{\fb}{\mathfrak b}
\newcommand{\fd}{\mathfrak d}
\newcommand{\ftw}{\mathfrak 2}
\newcommand{\fthr}{\mathfrak 3}
\newcommand{\cO}{{\mathcal O}}
\newcommand{\eps}{\varepsilon}
\newcommand{\Cl}{{\operatorname{Cl}}}
\newcommand{\Gal}{{\operatorname{Gal}}}

\newcommand{\lra}{\longrightarrow}

\newcommand{\la}{\langle}
\newcommand{\ra}{\rangle}

\newfont{\cyr}{wncyb10}
\newcommand{\TS}{\mbox{\cyr Sh}}

\newtheorem{thm}{Theorem}[section]
\newtheorem{prop}[thm]{Proposition}
\newtheorem{lem}[thm]{Lemma}

\numberwithin{equation}{section}

\title[Pure Cubic Number Fields]
       {Why is the Class Number of $\Q(\sqrt[3]{11})$ even?} 
\author{F. Lemmermeyer}
\email{hb3@ix.urz.uni-heidelberg.de}
\address{M\"orikeweg 1, 73489 Jagstzell, Germany}

\begin{document}

%\maketitle

\begin{abstract}
In this article we will describe a surprising observation that 
occurred in the construction of quadratic unramified extensions 
of a family of pure cubic number fields. Attempting to find an 
explanation will lead us on a magical mystery tour through the 
land of pure cubic number fields, Hilbert class fields, and 
elliptic curves.
\end{abstract}

\maketitle
\begin{center} \today  \end{center}

Euler was one of the (if not the) most prolific writers in mathematics.
Yet few if any of the articles appearing today are modeled after
Euler's way of writing; Euler often explained how he attacked a 
problem even if the attack ultimately proved unsuccessful: before
showing that an equation such as $x^3 + y^3 = z^3$ is not solvable in
integers he would try out one method of solving diophantine equations
after another. Gauss's motto ``pauca sed matura'' places him at the
other end of the spectrum: Gauss did not care very much about
conveying the motivation behind his proofs or about sketching the
paths that led him there, and Landau later wrote a whole series
of textbooks that consisted of little more than definitions, 
theorems and proofs.

Today, articles written in Euler's style have almost disappeared
from the literature for obvious economic (and other) reasons.
Here I would like to revive the Eulerian tradition and describe
in some detail the development from a curious observation on class
numbers to the results in Section \ref{S8}. What happened was that 
I numerically tested a result I wanted to use as an exercise for 
\cite{LS}; it turned out that the family of pure cubic number 
fields $\Q(\sqrt[3]{m}\,)$ for cubefree values of $m = 8b^3+3$ 
with $1 \le b < 89$ had even class numbers, but that the class 
number was odd for $b = 89$.

\section{An Exercise in Class Field Theory}

Consider pure cubic number fields $K = \Q(\sqrt[3]{m}\,)$ 
with $m = a^3+3$, and assume that $m$ is cubefree. The element 
$a - \omega$, where $\omega = \sqrt[3]{m}$, has norm 
$N(a - \omega) = a^3-m = -3$. Since $m \equiv 2, 3, 4 \bmod 9$,
the prime $3$ is ramified completely in $K/\Q$. Thus 
$(a-\omega)^3 = (3)$, and the element 
$\eps = -\frac13(a-\omega)^3 = 1 + a^2\omega - a\omega^2$ 
must be a unit in the ring of integers $\cO_K$.

If $a \equiv 0 \bmod 4$, this unit is positive and congruent to 
$1 \bmod 4$, hence $K(\sqrt{\eps}\,)/K$ is an unramified 
quadratic extension of $K$. 

\begin{prop}
Let $a > 0$ be an integer and assume that $m = a^3 + 3$ is cubefree. 
If $4 \mid a$ and $a \ne 0$, then the class number of 
$K =  \Q(\sqrt[3]{m}\,)$ is even.
\end{prop}

It only remains to show that the unit $\eps$ is not a square:

\begin{lem}
Let $a > 0$ be an integer and assume that $m = a^3 + 3$ is cubefree. 
Then $\eps =  1 + a^2\omega - a\omega^2$ is not a square in 
$K =  \Q(\sqrt[3]{m}\,)$. 
\end{lem}

\begin{proof}
From $\eps = \frac13(\omega-a)^3$ we see that if $\eps$ is a square
in $\cO_K$, then $3\omega - 3a = \beta^2$ is a square in $K$.
With $\beta = r + s\omega + t\omega^2$ we find the equations
$$ r^2 + 2st = -3a, \quad 2rs + mt^2 = 3 \quad \text{and} \quad 
   s^2 + 2rt = 0. $$
Since $m > 3$, these equations imply $st < 0$, $rs < 0$ and $rt < 0$:
but this is clearly impossible.
\end{proof}

If $a \equiv 2 \bmod 4$, let us write $a= 2b$ and $m = 8b^3 + 3$; 
the unit $\eps = 1 + 4b^2\omega - 2b\omega^2$ is not congruent to 
a square modulo $4$, but computing the class numbers of a few 
fields $K_b = \Q(\sqrt[3]{m}\,)$ produces the following results:

$$ \begin{array}{c|ccccccc}
 \rsp   b       &  1 &   3 &    5 &    7 &    9  &   11  &  13  \\  \hline    
 \rsp   m       & 11 & 3 \cdot 73 & 17 \cdot 59 & 41 \cdot 67 
                       & 3 \cdot 5 \cdot 389  &  10651 & 17579  \\
 \rsp h(K_b)    & 2  &  18 &   54 &  168 &  240  &  564  & 920
   \end{array} $$
Although these class numbers are all even, searching for a family of
explicit generators of unramified quadratic extensions of these 
cubic fields was unsuccessful.

Continuing this table shows that $h(K_b)$ is odd for 
$b = 19$; but here $m = 54875 = 5^3 \cdot 439$ is not cubefree.
The calculations have to be extended considerably before something
surprising happens:

$$ \begin{array}{c|ccccc}
 \rsp   b       &        85       &        86      &      87   
                     &       88      &   89  \\   \hline 
 \rsp   m       &  619 \cdot 7937 & 23^2 \cdot 9619 & 3 \cdot 1756009 
                     &  5451779      & 5 \cdot 11 \cdot 41^2 \cdot 61 \\
 \rsp h(K_b)    &   153954  & 6000  & 151200  &  186860 &  3375 
   \end{array} $$

Thus the class number of $K_b$ for $b = 89$ is odd, although it
is even for all 87 values less than $89$ for which $m = 8b^3+3$
is squarefree. This clearly cannot be an accident;
but how can we explain this phenomenon?

\section{Elliptic Curves}
Let $m = 8b^3 + 3$ for integers $b \ge 1$, assume that $m$
is cubefree, and let $K_b = \Q(\omega)$ denote the pure cubic 
number field defined by $\omega = \sqrt[3]{m}$. 

We now consider the family of elliptic curves $E_b: y^2 = x^3 - m$. 
A rational point on a curve $E_b$  has the form\footnote{For standard
results on the arithmetic of elliptic curves we refer to \cite{ST}
for a first introduction.} 
$x = \frac{r}{t^2}$ and $y = \frac{s}{t^3}$, and clearing denominators 
shows that such rational points correspond to solutions of the equation 
$s^2 = r^3 - m t^6$. In other words: the norm of the element 
$r - t^2 \omega \in K_b$ is a square. If this element is coprime 
to its conjugates, then there must be an ideal $\fa$ with 
$(r - t^2 \omega) = \fa^2$, and if $\fa$ is not principal, 
then the class number of $K_a$ will be even.

Computing a couple of rational points on these curves $E_b$ produces
the following table of selected points (here we have included even values
of $b$):

$$ \begin{array}{c|ccccc} 
    b &    1  &    2           &      3         &     4    & 5 \\   \hline
    P & (3,4) &  (17/4, 25/8)  &  (55/9, 82/27) &   (129/16, 193/64) 
               & (251/25, 376/125)
   \end{array} $$

These points all have the form $(f(b)/b^2, g(b)/b^3)$ for some 
(yet) unknown functions $f$ and $g$. Computing the differences 
we find
$$ \begin{array}{ccccccccccc}
      3  &    &  17 &    &  55 &    & 129 &     &  251 &     & 433 \\
         & 14 &     & 38 &     & 74 &     & 122 &      & 182 & \\
         &    &  24 &    &  36 &    & 48  &     &   60 &     & \\
         &    &     & 12 &     & 12 &     &  12 &      &     &
    \end{array} $$
The fact that the third differences seem to be constant and
equal to $12 = 2 \cdot 3!$ suggests that $f(b) = 2b^3 + $
terms of lower order, and then it is easy to guess that 
\begin{equation}\label{EPb}
 P_b\Big( \frac{2b^3+1}{b^2}, \frac{3b^3+1}{b^3} \Big) 
\end{equation}
is a family of rational points on the elliptic curves $E_b$. 
Since torsion points must be integral, we find

\begin{lem}
Let $a > 0$ be an integer and assume that $m = a^3 + 3$ is cubefree. 
The elliptic curves $E_b: y^2 = x^3 - m$ have rank at least $1$, and 
the rational points in (\ref{EPb}) have infinite order.
\end{lem}

Thus all curves $E_b$ have rank at least $1$. Does this explain
the fact that the class numbers of $K_b$ tend to be even? Before
we return to this question, let us describe an approach that could
have predicted the family of rational points $P_b$.

\section{Quadratic Fields}

The points $P_b = (x,y)$ from (\ref{EPb}) satisfy $y^2 = x^3-m$;
clearing denominators then gives
\begin{equation}\label{Eb}
 (3b^3+1)^2 + b^6 m = (2b^3+1)^3. 
\end{equation}
Since the elements $\tau = 3b^3+1 + b^3 \sqrt{-m}$ and 
$\tau' =  3b^3+1 - b^3 \sqrt{-m}$ generate coprime ideals 
in $k_b = \Q(\sqrt{-m}\,)$, there must be an integral ideal $\fa$ 
in $k_b$ with $(\tau) = \fa^3$. This suggests that the class numbers 
of the quadratic number fields $k_b$ should have a tendency to be 
divisible by $3$. Numerical experiments, however, reveal that this 
is not correct. In fact, the element $\tau$ with norm $(2b^3+1)^3$ 
is a cube since 
\begin{equation}\label{Ebc}
 3b^2+1 + b^3 \sqrt{-m} = \Big(\frac{1+\sqrt{-m}}{2}\Big)^3. 
\end{equation}

Thus we have seen

\begin{lem}
The element $\tau = 3b^3+1 + b^3 \sqrt{-m}$ in $\Q(\sqrt{-m}\,)$
is a cube. In particular, the ideal $\fa$ with $(\tau) = \fa^3$
is principal.
\end{lem}

This also means that the rational points $P_b$ on the 
family of elliptic curves $E_b$ could have been constructed 
in a rather trivial way: for $m = 8b^3+3$, taking the cube 
of the element\footnote{To be honest, this only worked because
the coefficient of $\sqrt{-m}$ is a cube; taking the third power 
of $\frac12(3+\sqrt{-m}\,)$, for example, does not work since 
its cube is $-9b^3 + (3-b^3) \sqrt{-m}$.} $\frac{1+\sqrt{-m}}2$ 
immediately gives (\ref{Ebc}) and thus the points $P_b$.

The ideal classes $[\fa]$ of order dividing $3$ in the quadratic 
number fields $k_b$ are all trivial. This begs the question 
whether the ideal classes of order dividing $2$ in the number 
fields $K_b$ deduced from (\ref{Eb}) are also trivial. This is
what we will look into next.

\section{Pure Cubic Fields}

Let us now write the equation (\ref{Eb}) in the form
$$ (3b^3+1)^2 = (2b^3+1)^3 - b^6 m  = N(2b^3+1 - b^2 \omega). $$
The element $\alpha = 2b^3+1 - b^2\omega$ has square norm; if it is 
the square of a principal ideal, it will not explain our
observations on class numbers. Playing around with elements of
small norm leads us to the observation\footnote{Observe that
$(4b^2+1)^3 = 64b^6 + 48b^4 + 12b^2 +1$ and $m^2 = 64b^6 + 48 b^3 + 9$;
in order to make the second terms vanish we only have to adjust the
coefficients slightly.} that 
$$ (4b^3+1)^3 - b^3 \cdot m^2 = 3b^3+1, $$
which shows that 
$$ N(\beta) =  3b^3+1 \quad \text{for} \quad \beta = 4b^3+1 - b \omega^2. $$
Thus there exist elements of norm $3b^3+1$; but is $\beta^2 = \alpha$?
The answer is no because
$$ \beta^2 = 16b^6 + 8b^3 + 1 + b^2(8b^3+3)\omega - (8b^4+2b)\omega^2. $$
A simple calculation, however, shows that 
$$ \eps \alpha = \beta^2, $$
where $\eps = 1 + 4b^2\omega - 2b\omega$ is the unit in $K_b$ we 
have started with. Thus $(\alpha)$ is the square of the principal 
ideal $(\beta)$, and so the ideal classes coming from the rational 
points $P_b$ are all trivial. It seems that we are back to square one.

\section{Back to Elliptic Curves}

The family of points $P_b$ on the elliptic curves $E_b$ shows that
these curves all have rank $\ge 1$. In fact, the curves $E_b$ for
small values of $b$ all have rank $\ge 2$: in fact, the Mordell-Weil
rank is $2$ for $1 \le b \le 90$ except for
\begin{itemize}
\item the values $b = 9, 17, 18, 20, 25, 53, 54, 67, 82, 87$ for which the 
      rank is $4$;
\item the value $b = 13$ for which the computation of the
      rank is complicated by the likely presence of a nontrivial
      Tate-Shafarevich group; here the rank is bounded
      by $2 \le r \le 4$, and the $2$-Selmer rank is even.
\item the values $b = 77$ and $a = 80$, for which the Selmer rank is $2$,
      but the second generator has large height. 
\item the values $b = 44$ and $a = 89$, for which the rank is $1$.
\item the values $b = 56, 68, 69, 86$, for which the rank is $3$.  
\end{itemize}

These results show that trying to construct two independent families of 
points on $E_b$ is bound to fail since there are examples of curves with
rank $1$. We therefore should show that the ranks of the curves
$E_b$ have a tendency to be even. This can be accomplished with the
help of the parity conjecture.

\subsection*{Parity Conjectures}
For formulating the various statements connected with the name 
partiy conjectures, let $r$ denote the Mordell-Weil rank of
an elliptic curve, and $R$ the analytic rank, that is, the order
of vanishing of the L-series of $E$ at $s = 1$. The conjecture
of Birch and Swinnerton-Dyer predicts that $r = R$.

The functional equation of the L-series of $E$ connects the values
at $s$ and $2-s$; the completed L-series $L^*$ satsifes the functional
equation 
$$ L^*_E(2-s) = w(E) L^*(s), $$
where $w(E) \in \{\pm 1\}$ is called Artin's root number. If 
$w(E) = -1$, then setting $s = 1$ in the functional equation 
implies $L(1) = 0$, and the Birch and Swinnerton-Dyer Conjecture 
predicts that $E$ has rank $\ge 1$. More generally, the parity 
conjecture states that $ (-1)^r = w(E)$. Performing a $p$-descent
on an elliptic curve provides us with the $p$-Selmer rank $r_p$ 
of $E$, and this rank differs from the rank $r$ of $E$ by an
even number if the Tate-Shafarevich group of $E$ is finite.

\begin{thm}
If $\TS(E)$ is finite, then the parity conjecture $(-1)^r = w(E)$ holds.
\end{thm}

The root number for elliptic curves with $j$-invariant $0$
was computed by Birch and Stephens; for our curves $E_b$
we find that (see Liverance \cite{Liv}) 
$$ w(E) = \prod_{p^2 \mid m} \Big(\frac{-3}p\Big). $$

Thus we expect that the rank $r_b$ of $E_b$ is even whenever 
$m$ is squarefree, and that it is odd if $r_b$ is divisible 
by the square of a unique prime $p \equiv 2 \bmod 3$. The 
values $b < 200$ for which this happens are 
$$ b = 44, 56, 68, 69, 86, 89, 94, 119, 169, 177, 194, $$
which agrees perfectly with our computations above.
Observe that the class number for $b = 419$ is even, and
that $m =  5^2 \cdot 11^2 \cdot 227 \cdot 857$ is divisible
by the square of two primes $p \equiv 2 \bmod 3$.

The only examples of pure cubic fields $K_b$ with odd class
numbers for $b < 1630$ and cubefree $m$ are 
\begin{align*}
   b & = 89, 119, 169, 177, 209, 369, 369, 503, 615, 661, 
         719, 787, \\
     & \quad  903, 1069, 1145, 1219, 1319, 1365, 1387, 1419, 1629. 
\end{align*} 
For all these $b$, the number $m = 8b^3+3$ is divisible by
exactly one prime $p \equiv 2 \bmod 3$. The many values of $b$
ending in $19$ are explained by the observation that $m$ is
divisible by $5^2$ if $m \equiv 19, 69 \bmod 100$.

Nothing so far prevents a pure cubic field $K_b$ from having an
odd class number if the rank of $E_b$ is even: the rational points
on $E_b$ give rise to ideals $\fa$ in $K_b$ whose squares are
principal, but there is no guarantee that $\fa$ is not principal.
Yet all available numerical evidence points towards the following

\medskip\noindent{\bf Conjecture 1.}
{\em Let $b \ge 1$ be an integer and assume that $m = 8b^3+3$ is
cubefree. If the class number of $K_b = \Q(\sqrt[3]{m}\,)$ is odd, 
then the rank of the elliptic curve $E_b: y^2 = x^3 - m$ is $1$.} 
\medskip

This conjecture implies, by the parity conjecture, the following, 
which does not even mention elliptic curves:

\medskip\noindent{\bf Conjecture 2.}
{\em Let $b \ge 1$ be an integer and assume that $m = 8b^3+3$ is cubefree. 
If the class number of $K_b = \Q(\sqrt[3]{m}\,)$ is odd,
then $m$ is divisible by an odd number of squares of primes 
$p \equiv 2 \bmod 3$.}
\medskip

A weaker formulation of the conjecture would be that the class
number of $K_b$ is even whenever $m$ is squarefree. Even this
weaker conjecture is unlike anything I would have expected. 
After all, the primes with exponent $1$ and $2$ in the prime 
factorization of $m$ change their roles when $m$ is replaced 
by $m^2$ (observe that $\Q(\sqrt[3]{m}\,) = \Q(\sqrt[3]{m^2}\,)$). 

\section{Nobody Expects the Spanish Inquisition}

In the conjectures above, the condition that $m$ be squarefree  
seemed quite surprising at first. This condition also occurs in 
the computation of an integral basis of pure cubic fields: it 
is well known that the ring of integers in $\Q(\sqrt[3]{m}\,)$
is given by $\Z[\sqrt[3]{m}\,]$ (such cubic fields are called 
monogenic) if and only if $m \not\equiv \pm 1 \bmod 9$ is 
squarefree. If, on the other hand, $m$ is divisible by the square 
of a prime $p$, then $\frac1p \sqrt[3]{m^2}$ is integral, and
the field is not monogenic.

But what should the form of an integral basis have to do with the
parity of the class number? Ten years ago, just about any number 
theorist you would have asked probably would have answered ``nothing!''
and would have quoted the heuristics of Cohen, Lenstra and Martinet
as supporting evidence. In fact, Cohen and Lenstra \cite{CL} gave an 
explanation of the numerical evidence for the distribution of class 
numbers of quadratic number fields based on certain heuristics; the main
idea was that it's not the actual size of a class group that matters 
but rather the size of its automorphism group. It is clear that
the prime $2$ behaves differently in quadratic extensions $k$ since 
the $2$-class group $\Cl_2(k)$ is far from being random: Gauss's genus
theory predicts its rank, and even the invariants divisible by $4$.

Cohen and Martinet \cite{CM1} then extended this project to class 
groups of extensions of higher degree.  As in the case of quadratic 
extensions there were ``bad primes'' $p$ for which the behaviour of 
the $p$-class group $\Cl_p(k)$ was not believed to be random; in 
\cite{CM1}, the prime $2$ was considered to be good for nonnormal
cubic extensions $k$ although it divides the degree of the normal 
closure of $k/\Q$. In particular, the probability that the class
number of $k$ is even was predicted to be about $p = 0.25932$.
In \cite{CM2}, however, the authors cast some doubt on their
earlier conjectures and asked whether the prime $2$ actually was
bad in this case. 

Finally Bhargava and  Shankar \cite{BS} proved the following result:
the average size of the $2$-class groups of complex cubic number 
fields ordered by discriminant (or height) is smaller than the 
corresponding average for monogenic cubic fields. Their result
agrees with the Cohen-Martinet prediction based on the assumption
that the prime $2$ is good for nonnormal cubic fields.

An observation suggesting a relation between monogenic rings and
the distribution of class groups can actually already be found in 
the article \cite{EFO}, where the authors computed the $2$-rank 
of pure cubic number fields $\Q(\sqrt[3]{m}\,)$ by studying the 
elliptic curves $E: y^2 = x^3 \mp m$ and remarked
\begin{quote}
The primes\footnote{More precisely: the fields $\Q(\sqrt[3]{p}\,)$.} 
$p \equiv \pm 1 \bmod 9$ have relatively small $2$-class numbers.
It seems that the reason for this is the fact that $A_k: y^2 = x^3 + k$
has a point of order $2$ in $\Q_3$ iff $k^2 \equiv 1 \bmod 9$ \ldots
\end{quote}

Let me add the remark that the fact that the $2$-class groups of 
the fields $K_b$ do not seem to be random does, of course, not 
imply that the prime $2$ is bad in the sense of Cohen-Martinet 
because the family of fields $K_b$ has density $0$.

\section{Why the class number of $\Q(\sqrt[3]{11}\,)$ is even}

Since we have started our tour with the question why the 
class number of $\Q(\sqrt[3]{11}\,)$ is even it is about 
time we provide an answer.

Consider the elliptic curve $E: y^2 = x^3 - m$ for some cubefree 
integer $m \equiv 3 \bmod 4$, and let $K = \Q(\omega)$ denote the 
pure cubic number field defined by $\omega = \sqrt[3]{m}$. If 
$P = (r/t^2, s/t^3)$ is a rational point with $t \equiv 0 \bmod 2$, 
then $s^2 = r^3 - mt^6$ shows that $\alpha = r - t^2 \omega \in K$ 
is congruent to $1 \bmod 4$; moreover, $(\alpha)$ is, as we will 
see below, an ideal square. Thus the field $K(\sqrt{\alpha}\,)$
is a quadratic unramified extension of $K$, and by class field
theory, $K$ has even class number.

Computing the generators of the Mordell-Weil group $E(\Q)$ of
the elliptic curve $E: y^2 = x^3 - 11$ we find the two points
$P = (3,4)$ and $Q = (15,58)$. The sum
$P + Q = (9/4, -5/8)$ has the desired form, hence 
$\alpha = 9 - 4\omega$ works. The minimal polynomial of $\sqrt{\alpha}$
is $f(x) = x^6 - 27x^4  + 243x^2 - 25$, and the discriminant of the 
number field generated by a root of $f$ is $3^6 11^4$; a ``smaller''
polynomial generating the same number field is 
$g(x) =  x^6 - 3x^5 + 9x^4 - 1$. Thus $K = \Q(\sqrt[3]{11}\,)$ 
has the unramified quadratic extension $K(\sqrt{9-4\omega}\,)$.  

$$ \begin{array}{c|c|ccc}
     \text{point}  & & \alpha & \fa_P & [\fa_P]\\   \hline
 \rsp   (3,4)      & P & 3 - \omega & \ftw_1^2 & 1 \\
 \rsp   (15,58)    & Q & 15 - \omega & \ftw_1 \cdot \mathfrak{29} & [\ftw] \\
 \rsp   (\frac94, -\frac58) & P+Q & 9 - 4 \omega & \mathfrak{5} & [\ftw] \\
 \rsp   (\frac{345}{64}, - \frac{6179}{512}) & 2P & 345 - 64 \omega & 
                 \mathfrak{37}_3 \cdot \mathfrak{167} & 1 \\
 \rsp   (\frac{51945}{13456}, \frac{10647157}{1560896}) & 2Q & 
          51945 - 13456 \omega & 
          \mathfrak{37}_1 \cdot \mathfrak{83} \cdot \mathfrak{3467} & 1 \\
 \rsp  (\frac{861139}{23409}, \frac{799027820}{3581577}) & 3P &  
        861139 - 23409\omega & 
        \ftw_1^2 \cdot \mathfrak{5} \cdot \mathfrak{23} \cdot \mathfrak{1737017}
        & 1  
    \end{array} $$
\smallskip

The prime $2$ splits into two prime ideals in $K$, namely
$\ftw_1$ with norm $2$ and $\ftw_2$ with norm $4$. The squares
of these ideals are principal: we have $\ftw_1^2 = (5+2\omega+\omega^2)$ 
and $\ftw_2 = (3+\omega-\omega^2)$.

The prime number $37$ splits into three prime ideals in $K$;
the prime ideals in the decomposition 
$(37) = \mathfrak{37}_1 \mathfrak{37}_2 \mathfrak{37}_3$ are determined
by the congruences
$\omega \equiv - 9 \bmod \mathfrak{37}_1$,
$\omega \equiv -12 \bmod \mathfrak{37}_2$, and
$\omega \equiv -16 \bmod \mathfrak{37}_3$.
The first two ideals are  nonprincipal, whereas $\mathfrak{37}_3$ is 
generated by $5 - 2 \omega$. This is compatible with our construction 
of the Hilbert class field of $K$: a computation of the quadratic
residue symbols shows that 
$$ \Big[\frac{9 - 4\omega}{\mathfrak{37}_1} \Big]_2  = 
  \Big(\frac{45}{37}\Big) = -1, \quad
  \Big[\frac{9 - 4\omega}{\mathfrak{37}_2} \Big]_2  = 
  \Big(\frac{57}{37}\Big) = -1, \quad
\Big[\frac{9 - 4\omega}{\mathfrak{37}_3} \Big]_2  = 
  \Big(\frac{73}{37}\Big) = + 1. $$
Thus only $\mathfrak{37}_3$ splits in the Hilbert class field.

Without going into details we remark that $\fa_{3Q}$ belongs to the
ideal class $[\ftw]$. This is compatible with the conjecture that
the map $P \to [\fa_P]$ is a homomorphism from $E(\Q)$ to $\Cl(K)[2]$.

% 
% 3Q --> 50491376191 - 22468511025\omega; exp = 1 + 1 + 0 + 0 + 1
%
\medskip

Similarly, for $b = 3$ and $m = 219$, the curve $E_3(\Q)$
is generated by $P = (55/9, 82/27)$ and $Q = (283/9, 4744/27)$,
and the sum $P+Q$  provides us with the quadratic unramified 
extension $K(\sqrt{\alpha}\,)$ for $\alpha = 115657 - 12996 \omega$. 
The minimal polynomial of $\sqrt{\alpha}$ is 
$f(x) = x^6 - 346971x^4 + 40129624947x^2 - 1066391672856409$,
a ``smaller'' polynomial whose root generates the same number field
is $g(x) = x^6 - 3x^5 + 21x^4 - 37x^3 + 126x^2 - 108x - 3$.

\section{Hilbert Class Fields via Elliptic Curves}\label{S8}

We will now show that this construction works whenever
$E$ has rank $\ge 2$:

\begin{thm}\label{T1}
Let $b$ be an odd integer such that $m = 8b^3+3$ is squarefree.
Then the class number of $\Q(\sqrt[3]{m}\,)$ is even whenever
$E: y^2 = x^3 - m$ has rank $\ge 2$. If the parity conjecture
holds, then the class number is even for all squarefree values 
of $m$. 
\end{thm}

This theorem will be proved by showing that if there is a point
$P = (r/t^2, s/t^3)$ on $E(\Q) \setminus 2E(\Q)$ with $2 \mid t$, 
then the extension $H = K(\sqrt{\alpha}\,)$ of $K$ is a quadratic 
unramified extension. 

For showing that $H/K$ is unramified we have to verify the 
following claims:
\begin{itemize}
\item $\alpha > 0$, which implies that the extension $H/K$
      is unramified at the infinite primes;
\item $\alpha \equiv 1 \bmod 4$, which implies that $H/K$ is
      unramified above $2$;
\item $(\alpha) = \fa^2$ is an ideal square, which implies that
      $H/K$ is unramified at all finite primes not dividing $2$.
\end{itemize}
The first claim is trivial since $N\alpha = s^2 > 0$, and the
second claim follows from the assumption $2 \mid t$. It remains
to show that $(\alpha)$ is an ideal square:

\begin{lem}\label{Lcop}
Let $P = (r/t^2, s/t^3)$ be a rational point on 
$E_b: y^2 = x^3 - m$ for a squarefree value of $m = 8b^3+3$.
Assume as above that $\gcd(r,t) = \gcd(s,t) = 1$. 
Then the ideal $(\alpha)$ for $\alpha = r - t^2\omega$ 
is the square of an ideal $\fa$ in $K_b$.
\end{lem}

\begin{proof}
The ideal $(\alpha)$ is a square if $N\alpha$ is a square and
$(\alpha,\alpha') = (1)$ in $K_b' = \Q(\sqrt{-3},\omega)$. In our 
case, $N\alpha = s^2$, and any ideal divisor of $\alpha$ and 
$\alpha'$ divides the difference $\alpha-\alpha' = (1-\rho)t^2\omega$.
Since $\gcd(s,t) = 1$, this ideal must divide $(1-\rho)\omega$. 
Any prime ideal dividing $\omega$ and $s$ also divides $r$, so its
norm divides both $r$ and $s$, hence the square of its norm divides 
$mt^6$. Since $\gcd(r,t) = 1$, it must divide $m$, and this contradicts
the assumption that $m$ be squarefree.

Thus the only possibilities for $\fd = (\alpha,\alpha')$ are
$\fd = (1)$ and $\fd = \fthr$, where $\fthr$ is the prime ideal
above $3$ (recall that $m \equiv 2, 3, 4 \bmod 9$, hence 
$3\cO_K = \fthr^3$). The second case is only possible if $3 \mid s$,
but this leads quickly to a contradiction, since in this case
$r$ and $t$ are not divisible by $3$, and $r^2 - mt^6$ is not
divisibly by $9$ in this case. Thus $(\alpha) = \fa^2$ is an ideal square.
\end{proof}

For showing that $H/K$ is a quadratic extension we need to 
know when $\alpha$ is  square in $K$. To this end let us first 
characterize the points on $E$ that give rise to squares:

\begin{lem}
Let $m$ be a cubefree integer, $K = \Q(\omega)$ the corresponding
pure cubic number field with $\omega^3 = m$, and $E: y^2 = x^3-m$ 
an elliptic curve. Every rational affine point $P \in E(\Q)$ can 
be written in the form $P = (r/t^2,s/t^3)$ for integers $r, s, t$ 
with $\gcd(r,t) = \gcd(s,t) = 1$.

The map $\alpha: E(\Q) \lra K^\times/K^{\times\,2}$ defined by
$\alpha(P) = (r-t^2\omega) K^{\times\,2}$ is a group homomorphism
whose kernel contains $2E(\Q)$; more exactly $\alpha(2P)$ is 
represented by the square of $\beta = (r-t^2\omega)^2 - 3(t^2\omega)^2$. 

Finally if $P \in \ker \alpha$ and $t$ is even, then $P = 2Q$ for 
some $Q \in E(\Q)$.
\end{lem}

\begin{proof}
For a proof that we may assume $\gcd(r,t) = \gcd(s,t) = 1$
see \cite{ST}.

Performing a $2$-descent on the elliptic curve $E: y^2 = x^3 - m$
means studying the Weil map $E(K) \lra  K^\times/K^{\times\,2}$ which
sends a $K$-rational point $P = (x,y)$ to the coset represented by 
$x - \omega$. The fact that the Weil map is a homomorphism is 
classical (see e.g. \cite{ST}); in particular, the restriction of 
the Weil map to $E(\Q)$ is also a homomorphism.

Since the target group is $K^\times$ modulo squares, the Weil map can 
be defined by $\alpha(P) = (r - t^2\omega)K^{\times\,2}$. 

Now assume that $\alpha(P) = (r - t^2\omega)K^{\times\,2}$ and set 
$\beta = \alpha(P)^2 - 3t^4\omega^2$. Then
$$ \beta^2 = (r^2 - 2rt^2\omega - 2t^4 \omega^2)^2 
           = r^4 + 8rt^6m  - 4t^2 \omega(r^3 - mt^6). $$
On the other hand, the group law on $E$ gives
$$ 2(x,y) = \Big( \frac{9x^4}{4y^2} - 2x, 
                  - \frac{27x^6}{8y^3} + \frac{9x^3}{2y} - y\Big) 
          = \Big( \frac{x^4 + 8mx}{(2y)^2} , 
                    \frac{x^6 - 20mx^3 - 8m^2}{(2y)^3} \Big). $$
Thus
$$ x_{2P} =  \frac{x^4 + 8mx}{(2y)^2} 
                  =  \frac{x^4 + 8mx}{4x^3-4m}
                  =  \frac{r^4/t^8 + 8mr/t^2}{4r^3/t^6  - 4m}
                  =  \frac{r^4 + 8rmt^6}{4t^2(r^3 - mt^6)}, $$
and this implies the claim.

Finally assume that $\alpha(P) \in K^{\times\,2}$ for $P = (x,y)$. Since 
$$ y^2 = x^3-m = (x-\omega)(x-\rho\omega)(x-\rho^2\omega), $$
we find that 
$$ x^2 + x\omega + \omega^2 
            = (x-\rho\omega)(x-\rho^2\omega) \in K^{\times\,2}. $$ 

The ideals $(x-\rho\omega)$ and $(x-\rho^2\omega)$ are coprime
in $L = K(\sqrt{-3}\,)$ by the proof of Lemma \ref{Lcop}, hence 
\begin{equation}\label{Ert}
 r - \rho\omega t^2 =  \eta \beta^2 
\end{equation}
for some unit $\eta$ in $L$. Now we need a variant of 
Kummer's Lemma for $L$:

\begin{prop}\label{Punit}
If $\eta$ is a unit in $L$ with $\eta \equiv \xi^2 \bmod 4$,
then $\eta$ is a square.
\end{prop}

Since the left hand side of (\ref{Ert}) is congruent to $1 \bmod 4$,
Prop. \ref{Punit} implies that $\eta$ is a square. But then 
$r - \rho\omega t^2$ (and therefore also $r - \rho^2\omega t^2$)
is a square. 
By \cite[Prop. 1.7.5]{Conn}, we have $P = (x_P, y_P) = 2Q$ for some 
point $Q \in E(\Q)$ on $E: y^2 = f(x) = (x-e_1)(x-e_2)(x-e_3)$ 
if and only if $x_P - e_j$ is a square in $\Q(e_j)$ for $j = 1, 2, 3$.
This implies our claims\footnote{A simpler proof that $P = 2Q$ in 
this case can be found in \cite{L3}.}.
\end{proof}

If $E$ has rank $\ge 2$, let $P$ and $Q$ denote two generators of
$E(\Q)$. If one of them has the desired form, we are done. If not,
then we claim that $P+Q$ works:

\begin{lem}
If $P$ and $Q$ are independent points with odd denominators,
then $P+Q = (r/t^2, s/t^3)$ with $\gcd(r,t) = \gcd(s,t) = 1$ and 
$2 \mid t$.
\end{lem}

\begin{proof}
This is a direct consequence of the addition formulas. In fact, we 
find $x_3 = \mu^2 -x_1-x_2$ for $\mu = \frac{y_2-y_1}{x_2-x_1}$. 
Now $(y_1-y_2)(y_1+y_2) = y_1^2-y_2^2 = x_1^3 - x_2^3 = 
 (x_1-x_2)(x_1^2 + x_1x_2 + x_2^2)$; we know by assumption that
$x_1 \equiv x_2 \equiv 1 \bmod 2$, hence the right hand side
is divisible by $2$. Since the second bracket on the right side
is odd, the whole power of $2$ is contained in $x_1-x_2$.
On the left hand side, the power of $2$ is split among the factors
$y_1-y_2$ and $y_1+y_2$, both of which are even. This implies that
the denominator of $\mu$ must be even. In particular, the denominator
of $x_3 = \mu^2 - x_1-x_2$ must also be even, which is what we wanted
to prove.
\end{proof}

It remains to give a 

\begin{proof}[Proof of Prop. \ref{Punit}]
We start by determining the unit group of $L$. We know that 
$E_K = \la -1, \eps \ra$; it is easy to verify that 
$E = \la -\rho, \eps, \eps' \ra$ has finite index in $E_L$
(for example by showing that the regulator of this group
is nonzero), where $\eps' = 1 + 4b^2\rho\omega - 2b\rho^2\omega^2$
is the conjugate of $\eps$ over $K$. We first show that $E = E_L$.

\begin{itemize}
\item The units $\pm \eps$ are not squares in $L$. In fact,
      if $\pm \eps$ is a square in $L$, then  
      $L = K(\sqrt{\pm \eps}\,)$. This implies that $L/K$
      is unramified outside $2 \infty$, which contradicts the
      fact that $L/K$ is ramified above $3$.
\item The units $\pm \rho^c \eps$ are not squares in $L$: this
      follows from above by noting that $\rho$ is a square in $L$.
\item The units $\pm \rho^c \eps'$ and  $\pm \rho^c \eps''$ are not 
      squares in $L$: this follows by applying a suitable automorphism
      of $\Gal(L/\Q)$.  
\item The units $\pm \rho^c \eps'\eps''$ are not squares in $L$: 
      this follows from the above by observing that $\eps'\eps'' = 1/\eps$.
\end{itemize}
This shows that $E$ has odd index in $E_L$. The fact that the index 
must then be $1$ follows by applying the norm $N_{L/K}$ to any relation 
of the form $(-\rho)^a \eps^b {\eps'}^c = \eta^p$.

Now assume that $\eta \equiv \xi^2 \bmod 4$ for some unit $\eta \in E_L$.
Since $\rho$ is a square, we may assume that 
$\eta = (-1)^a \eps^b  {\eps'}^c$ with $a, b, c \in \{0, 1\}$. We know
that $\eta \equiv 1 \bmod 2$; if $\eta \equiv \xi^2 \bmod 4$, we must
have $\eta \equiv 1 \bmod 4$. Checking the finitely many possibilities
we easily deduce that $a = b = c = 0$, and this implies the claim.
\end{proof}

This completes the proof of Theorem \ref{T1}. 
As a matter of fact, we can prove something stronger:

\begin{thm}
Let $b$ be an odd integer such that $m = 8b^3+3$ is squarefree.
Then the $2$-rank $s$ of the class group $\Cl_2(K)$ of
$K = \Q(\sqrt[3]{m}\,)$ and the rank $r$ of the Mordell-Weil
group $E(\Q)$ of $E: y^2 = x^3 - m$ satisfy the inequality
\begin{equation}\label{Einrs} r \le s + 1. \end{equation}
\end{thm}

\begin{proof}
Let $P_1, \ldots, P_k$ denote the generators of $E(\Q)$
with even denominators, and $P_{k+1}, \ldots, P_r$ those with
odd denominators. Then the points 
$P_1$, \ldots, $P_k$, $P_{k+1}+P_r, \ldots, P_{r-1}+P_r$ are
independent points in $E(\Q) \setminus 2E(\Q)$ with even
denominators; by what we have proved, the pure cubic field
$K$ has $r-1$ independent unramified quadratic extensions.
\end{proof}

In the only case where we have been unable to compute the rank
of the curve $E_b$, namely for $b = 13$, we find $s = 3$ and
therefore $r \le 4$. This does not improve on the bound coming
from the Selmer rank, but it suggests that the rank $r$ of $E(\Q)$
in the preceding theorem may perhaps be replaced by some Selmer 
rank.

An inequality similar to (\ref{Einrs}) for general elliptic curves
was given by Billing (see \cite[Sect. 3.7]{Conn}). For curves
$E: y^2 = x^3 - m$, Billings bound was
$$ r \le \begin{cases}
                 s + 1 & \text{ if } m \not \equiv \pm 1 \bmod 9, \\
                 s + 2 & \text{ if } m      \equiv \pm 1 \bmod 9.
          \end{cases} $$
The main difference between Billing's result applied to $E_b$ and ours 
is that Billing proved $r \le s + 1$, whereas we proved $s \ge r-1$
by more or less explicitly exhibiting generators of the quadratic
unramified extensions of $K$.

\section*{Additional Remarks and Open Problems}
The explanation that the fields $K_b$ tend to have even class numbers
could have been giving by simply citing the parity conjecture and
Billing's bound. We have shown more by using rational points on 
elliptic curves for constructing subfields of Hilbert class fields.
It remains to be studied how much of Billing's bound can be proved
in a similar way.

Instead of using the curves $E_m: y^2 = x^3 - m$ we could also investigate
the family $E_{-m}: y^2 = x^3 + m$; in this case, the root number is 
$w(E_{-m}) = - w(E_m)$, so we expect that its rank is odd whenever
$m$ is squarefree. For the curves with rank $1$, the generator
seems to have even denominator in most cases. It also seems that
the curves $E_{-m}$ more often have nontrivial Tate-Shafarevich groups
than the curves $E_m$, but this might be a general feature of
families of elliptic curves with root number $-1$ when
compared with families of curves with rank $\ge 1$ and positive
root number.

I would like to call the attention of the readers to the fact
that Soleng \cite{Sol} has constructed a homomorphism from the group 
of rational points on elliptic curves to the class groups of certain 
quadratic number fields. Our map sending the points $P = (r/t^2,s/t^3)$
with even $t$ to the ideal class $[\fb]$, where 
$(s + t^3\sqrt{-m}\,)  = \fb^2$, does not seem to be a special case
of Soleng's construction. Is this map also a homomorphism? How are
these maps related to the group structure on Pell surfaces 
$y^2 + mz^2 = x^3$ studied in \cite{HL}?

Our calculations in the Mordell-Weil group of $y^2 = x^3 -11$ and the
corresponding pure cubic field $K = \Q(\sqrt[3]{11}\,)$ paired with
a strong belief in the prestabilized harmony of algebraic number theory
and a considerable amount of wishful thinking suggest the following:
If $P = (x,y)$ with $x = \frac{r}{t^2}$ is a rational point on 
the elliptic curve $E:  y^2 = x^3 - m$, then under suitable conditions 
on $m$, the ideal $(r - t^2 \omega) = \fa_P^2$ is an ideal square in 
$K = \Q(\sqrt[3]{m}\,)$, and  the map sending $P$ to the ideal class 
of $\fa_P$ is a homomorphism from $E(\Q)$ to the $2$-class group
$\Cl(K)[2]$ of the pure cubic field $K$; this is certainly a problem 
that deserves to be studied.

Paul Monsky gave a proof that the class number of $\Q(\sqrt[4]{p}\,)$
is even for primes $p \equiv 9 \bmod 16$ based on the parity 
conjecture in \cite{Mon1}; an unconditional proof of this fact can 
be found in \cite{Mon2}. Is it possible to give an unconditional 
proof of the results on the parity of the class numbers of the 
pure cubic fields $K_b$?

\section*{Acknowledgement}

I thank Dror Speiser \cite{MO} for reminding me of the approach using 
elliptic curves and for pointing out the relevance of \cite{Liv}. 
Paul Monsky kindly sent me an unpublished manuscript \cite{Mon1} 
in which he studied connections between the parity of class numbers 
of pure quartic fields and elliptic curves. 
All calculations were done with {\tt pari} and {\tt sage}.

\vskip 0.5cm

\end{document}